\documentclass[a4paper,11pt,reqno]{amsart}

\usepackage{amsmath,amssymb,amsthm} 
\usepackage{graphics} 
\usepackage{neuralnetwork}
\usetikzlibrary{shapes}
\usepackage{tikz}

\usepackage[colorlinks = true]{hyperref}

\usepackage{cleveref}

\usepackage[hyperref,doi=false,url=false,isbn=false, giveninits=true, backref=true,style=alphabetic,maxbibnames=99, backend=biber]{biblatex}
 \bibliography{deep_learning.bib}

 \usepackage{bm}

\numberwithin{equation}{section}

\newtheorem{theorem}{Theorem}[section]
\newtheorem{lemma}[theorem]{Lemma}
\newtheorem{proposition}[theorem]{Proposition}

\theoremstyle{definition}

\newtheorem{remark}[theorem]{Remark}

\newtheorem{example}[theorem]{Example}

\newcommand{\R}{\mathbb{R}}

\newcommand{\eps}{\varepsilon}

\newcommand{\ov}{\overline}

\newcommand{\be}{\begin{equation}}
\newcommand{\ee}{\end{equation}}

\newcommand{\sym}{sym}
\makeatletter
\renewcommand{\fnum@figure}{Fig. \thefigure}
\makeatother
\newcommand{\cost}{\mathsf{c}}
\def\Id{{\rm Id}}
\def\tr{{\rm tr}}

\newcommand{\out}{\operatorname{out}}
\newcommand{\din}{\operatorname{in}}

\def\EE{\mathbb{E}}

\newcommand{\op}{\operatorname{op}}

\newcommand{\bra}[1]{\left( #1 \right)}
\newcommand{\sqa}[1]{\left[ #1 \right]}
\newcommand{\cur}[1]{\left\{ #1 \right\}}

\newcommand{\abs}[1]{\left| #1 \right|}
\newcommand{\nor}[1]{\left\| #1 \right\|}

\usepackage{enumerate}
\newcommand{\Lip}{\operatorname{Lip}}

\renewcommand{\sym}{\operatorname{sym}}

\newcommand{\cX}{\mathcal{X}}

\title[Functional LDP for wide Gaussian Lipschitz DNN]{Functional Large Deviations for Wide Deep Neural Networks with Gaussian Initialization and Lipschitz Activations}

\author[C. Macci]{Claudio Macci}
\address{C.M.: Dipartimento di Matematica, Università di Roma Tor Vergata, Via della Ricerca Scientifica, I-00133
Rome, Italy.}
\email{macci@mat.uniroma2.it}

\author[B. Pacchiarotti]{Barbara Pacchiarotti}
\address{B.P.: Dipartimento di Matematica, Università di Roma Tor Vergata, Via della Ricerca Scientifica, I-00133
Rome, Italy.}
\email{pacchiar@mat.uniroma2.it}

\author[K. Papagiannouli]{Katerina Papagiannouli}
\address{K.P.: Dipartimento di Matematica, Università degli Studi di Pisa, Largo Bruno Pontecorvo 5, I-56127 Pisa, Italy.}
\email{aikaterini.papagiannouli@unipi.it}

\author[G. L. Torrisi]{Giovanni Luca Torrisi}
\address{G.L.T.: Istituto per le Applicazioni del Calcolo, Consiglio Nazionale delle Ricerche, Via dei Taurini 19, 00185
Rome, Italy.}
\email{giovanniluca.torrisi@cnr.it}

\author[D. Trevisan]{Dario Trevisan}
\address{D.T.: Dipartimento di Matematica, Università degli Studi di Pisa, Largo Bruno Pontecorvo 5, I-56127 Pisa, Italy.
}
\email{dario.trevisan@unipi.it}

\subjclass[2010]{60F10, 60F05, 68T07}
\keywords{Neural Networks, Large Deviations, Cramér Theorem}

\begin{document}

\begin{abstract}
We establish a functional large deviation principle for fully connected multi-layer perceptrons with i.i.d.\ Gaussian weights (LeCun initialization) and general Lipschitz activation functions, including therefore the popular case of ReLU. The large deviation principle holds for the entire network output process on any compact input set. The proof combines exponential tightness for recursively defined processes, finite-dimensional large deviations, and 
the Dawson–Gärtner theorem, extending existing results beyond 
finite input sets and less general activations.
\end{abstract}

\maketitle
\section{Introduction}

Large deviation principles (LDPs) provide a natural framework to quantify non-typical behavior in large stochastic systems. In the context of random neural networks, recent years have seen substantial progress in understanding typical fluctuations, Gaussian limits, and mean-field descriptions in the large-width regime; see, among many others, \cite{neal_priors_1996,lee_deep_2018,matthews_gaussian_2018,hanin_random_2021,yang_wide_2019}. By contrast, the study of large deviations for deep neural networks is still in an early stage.

Recent works have established LDPs for observables of randomly initialized fully connected feedforward neural networks, also called multi-layer perceptrons (MLPs), focusing on Gaussian weights (LeCun initialization) and activation functions which exclude important nonlinearities used in practice, most notably the ReLU activation; see \cite{macci2024large} and \cite{andreis2025ldp}. This was partially overcome in \cite{vogel2025large}, where an LDP is stated for finite-dimensional distributions of MLPs with linear growth activations; however, a functional LDP, i.e., an LDP at the level of the entire network output process is still missing in this setting. The purpose of this work is to fill this gap. 

\subsection{Main result and outline of its proof}

The statement of our main result can be given as follows. We refer the reader to Section \ref{sec:notation} for notation and definitions.

\begin{theorem}\label{thm:main}
Fix a depth $L \ge 1$, input $d_{\din}$, and output  $d_{\out}\ge 1$ dimensions, Lipschitz activation functions $(\sigma^{(\ell)})_{\ell=1,\ldots, L}$, and consider a sequence of outputs $f^{(L)}_{(n)}: \R^{d_{\din}} \to \R^{d_{\out}}$ of a fully connected feedforward neural network (MLP) with hidden layer widths $n_\ell = n_\ell(n)$ such that 
$$\lim_{n\to \infty} \frac{n_\ell}{n} \in (0, \infty]  \quad \text{ for every $\ell=1, \ldots, L-1$,} $$
and i.i.d.\ Gaussian weights as in \eqref{eq:lecun} --  also known as LeCun initialization. 

Then, for every compact input set $\cX \subseteq \R^{d_{\din}}$, the rescaled output of the network
$$ x \in \cX \mapsto \frac{1}{\sqrt{n}} f^{(L)}_{(n)}(x), $$
seen as a random element on $C(\cX,\mathbb{R}^{d_{out}})$, endowed with the topology of uniform convergence,
satisfies the LDP with speed $n$ and a suitable rate function (see \eqref{eq:rate-functional}).
\end{theorem}

The proof relies on an inductive argument on the number of layers, combining Cram\'er's theorem for empirical means of i.i.d.\ random variables with a careful analysis of the continuity properties of the neural network architecture. The high-level structure is somewhat analogous to that of the functional CLT for MLPs presented in \cite{bracale_large-width_2020}, of course, with modifications to account for the different nature of large deviations. In brief:
\begin{enumerate}
\item[(i)] In Section~\ref{sec:exp-tight} we establish exponential tightness of the sequence $(\frac{1}{\sqrt{n}} f^{(L)}_{(n)}(\cdot))_{n\geq 1}$ on $C(\cX,\mathbb{R}^{d_{out}})$, endowed with the uniform topology.
In particular, we use a variant of Schied's criterion \cite{schied1994criteria,schied1997moderate}, Theorem~\ref{thm:schied}, tailored to the layer-by-layer propagation of randomness and taking into account the sub-linear growth of activation functions. 

\item[(ii)] In Section \ref{sec:ldp-finite} we first prove that for a fixed finite set of inputs, the LDP for the corresponding finite-dimensional network outputs holds. This step relies on Chaganty's lemma \cite{chaganty1997large}, combined with exponential equivalence arguments, and Cram\'er's theorem under local exponential moment assumptions. The overall strategy is close in spirit to the recent contribution by Vogel \cite{vogel2025large}, although our argument crucially relies on epi-convergence techniques from convex analysis.\footnote{Precisely, we exploit epi-convergence techniques to establish the joint lower semi-continuity of a suitable function, filling a gap in Vogel's argument contained in \cite[Lemma 2.2]{vogel2025large}, and precisely implication (2.11) therein. To argue that this implication can not hold in general, one considers e.g. the scalar (one-dimensional) case, $\sigma(x)=\max\{0,x\}+\eta$, $x\in\mathbb R$, $\eta\neq 0$, $q=0$ and $q_n>0$ for every $n$. Then, for any fixed $\varepsilon>0$ and $n$, for $\eta$ sufficiently small the inequality (2.11) in \cite{vogel2025large} is false on an event of positive probability.} 

\item[(iii)] To conclude (see the argument after the proof of Proposition \ref{prop:ldp-kernel}), we lift the finite-input LDPs to an LDP for the entire process by means of a standard application of the Dawson--G\"artner theorem and the inverse contraction principle.

\end{enumerate}

\subsection{Relation to the literature.}

The work of three of the authors of this paper \cite{macci2024large}, primarily established an LDP for the output of MLPs with bounded activations and finite input sets. Functional
LDPs for MLPs obtained so far have been limited to settings with Gaussian weights and growth assumptions on the activation function, which exclude the ReLU \cite{andreis2025ldp}. The recent work \cite{vogel2025large} addresses large deviations for Gaussian networks with ReLU activation, but focuses on a finite set of inputs. To the best of our knowledge, a functional LDP for deep neural networks with general Lipschitz activations has not been previously established.

\subsection{Perspectives}

The results of this work suggest several natural directions for future research. First, the recursive  structure of rate functions is compatible with more general architectures and could potentially be combined with the tensor-program framework of \cite{yang_wide_2019} to treat convolutional, residual, or attention-based networks.  

Second, the probabilistic techniques developed in this paper provide a natural starting point for the study of large deviations during training, where the dynamics of the parameters introduce additional temporal dependence and non-equilibrium effects. Establishing LDPs for training dynamics remains a largely open problem; in this context, some specific aspects, in the case of one-hidden-layer neural networks, are investigated in \cite{HirschWillhalm}.

\section{Notation and basic facts} \label{sec:notation}

\subsection{General notation}\label{sec:gennot}

Positive constants are generally denoted by $\cost$, and with $\cost(\ldots)$ to indicate dependence on parameters; their value may occasionally change from line to line. We use the measure-theoretic convention that $\infty \cdot 0 = 0$, denote by $\|\cdot\|_{L^p}$ the usual Lebesgue norm with respect to a probability measure (which will be clear from the context) and by $\overset{\it{law}}{=}$ the equality in law between random variables.

Given $d \ge 1$, we write $\R^d$ for the $d$-dimensional Euclidean space, with norm denoted by $|\cdot|$. Given a matrix $A \in \R^{m \times n}$, we denote by $A^\top$ its transpose and by $\nor{A}_{\op}$ its operator norm (induced by the Euclidean norms on $\R^n$ and $\R^m$); symmetric $d\times d$ matrices are denoted by $\R^{d \times d}_{\sym}$ and positive semidefinite (symmetric) matrices by $\R^{d \times d}_{\succeq 0}$. Every $q \in \R^{d\times d}_{\succeq 0}$ admits a unique (non-negative) square root, i.e., $\sqrt{q} \in \R^{d \times d}_{\succeq 0}$ such that $q = (\sqrt{q})^2$. Its construction uses the spectral theorem, reducing to the simpler case of diagonal matrices. The map $q \mapsto \sqrt{q}$ is H\"older continuous with exponent $1/2$ (see e.g.\ \cite{van1980inequality}),  i.e., 
$$ \nor{\sqrt{q} - \sqrt{q'} }_{\op} \le \sqrt{ \nor{q -q'}_{\op} }.$$
If  $q\in \R^{d\times d}_{\succeq 0}$ is invertible, we write $q^{-1}$ for its inverse, which induces the quadratic form $z^\top q^{-1} z$, for $z \in \R^d$. We extend this form to general matrices $q\in \R^{d\times d}_{\succeq 0}$ in the following way: if $z=qx$ belongs to the image of $q$, then we define $\nor{z}_q^2 := z^\top x$, otherwise we set $\nor{z}_q^2 = \infty$. Such a function clearly extends $z^{\top}q^{-1} z$ when $q$ is not invertible and takes values in $[0, \infty]$. This quadratic form coincides with the squared norm of the reproducing kernel Hilbert space (Cameron--Martin space) associated with the Gaussian law $\mathcal N(0,q)$, extended by $+\infty$ outside $\mathrm{Im}(q)$, see e.g. the monograph \cite{Bogachev1998}. We prove in Remark~\ref{rem:zqz-lsc} that $(q,z) \mapsto \nor{z}_q^2$ is jointly lower semicontinuous. 

Given $A \in \R^{d \times d}$, we denote by $\tr(A) = \sum_{i=1}^{d} A_{ii}$ its trace. The trace induces the Hilbert-Schmidt scalar product on $\R^{d \times d}$, defined as $(A,B)\mapsto \tr(A^\top B)$, and the norm $\nor{A} = \sqrt{ \tr(A^\top A)}$, which we use in particular to identify $\R^{k \times k}_{\sym}$ with $\R^{k(k-1)/2}$. Given $v, w \in \R^d$, we write $v\otimes w = v w^\top \in \R^{d\times d}$ for the rank-one matrix $(v\otimes w)_{ij}= v_i w_j$, for which $\nor{v\otimes w}_{\op} = \nor{v \otimes w} =\abs{v} \abs{w}$. We also write for brevity $v^{\otimes 2} = v \otimes v \in \R^{d\times d}_{\sym}$.

Lipschitz continuous functions $\sigma: \R \to \R$ are such that, for some constant $\Lip(\sigma)$,
$$ |\sigma(x) - \sigma(y)| \le \Lip(\sigma) |x-y|, \quad \text{for every $x$, $y \in \R$,} $$
and are functions with at most linear growth: we set
$$ \nor{\sigma}_{\Lip} := \abs{\sigma(0)} + \Lip(\sigma),$$
so that $\abs{\sigma(x)} \le \nor{\sigma}_{\Lip}(1 + \abs{x})$ for every $x \in \R$.

\subsection{Convex analysis} For later purposes we recall some basic facts from convex analysis; we refer to \cite{rockafellar1998variational} for details. Given a function $f: \R^d \to [-\infty, \infty]$, its convex conjugate (or Legendre-Fenchel transform) $f^*: \R^d \to [-\infty, \infty]$ is defined by 
$$ f^*(v) := \sup_{x \in \R^d} \bra{ v^\top x - f(x) }, \quad \text{for every $v \in \R^d$.}$$
A function is said to be proper if it is not identically equal to $+\infty$ and never attains the value $-\infty$. If $f$ is proper, convex and lower semicontinuous, then $f^*$ is also proper, convex, and lower semicontinuous.

A useful notion of convergence to deal with sequences of proper, convex, and lower semicontinuous functions is given by epi-convergence: $(f_n)_{n=1}^\infty$ are said to epi-converge to $f: \R^d \to [-\infty, \infty]$ if the following two conditions hold:
\begin{equation}\label{eq:liminf}
 \liminf_{n \to \infty} f_n(x_n) \ge f(x), \quad \text{for every $x \in \R^d$ and every sequence $(x_n)_{n}$ such that $x_n \to x$,}
 \end{equation}
 and
 \begin{equation}\label{eq:limsup}
  \limsup_{n \to \infty} f_n(x_n) \le f(x), \quad \text{for every $x \in \R^d$ and some sequence $(x_n)_{n}$ such that $x_n \to x$.}
  \end{equation}
A key result that is useful for our purposes is the following, see \cite[11.34 Theorem]{rockafellar1998variational}.

\begin{theorem}[epi-continuity of Legendre-Fenchel transform]\label{thm:wijsman}
Let $(f_n)_{n}$, $f$ be proper, lower semicontinuous and convex functions on $\R^d$. One has that $(f_n)_n$ epi-converges to $f$ if and only if  $(f_n^*)_n$ epi-converges to $f^*$.
\end{theorem}

\begin{remark}\label{rem:zqz-lsc}
Using Theorem \ref{thm:wijsman}, it is possible to argue that the function 
$$ \R^d \times \R^{d\times d}_{\succeq 0} \ni (z,q) \mapsto \nor{z}_q^2 \in[0, \infty]$$
(defined in the extended sense introduced in Section \ref{sec:gennot} if $q$ is not invertible) is lower semicontinuous. Given a sequence $((z_n, q_n))_{n=1}^{\infty}\subset \R^d \times \R^{d\times d}_{\succeq 0} $ converging to $(z,q)\in\R^d \times \R^{d\times d}_{\succeq 0}$, consider the proper, convex, and continuous functions $f_n(x) := x^\top q_n x$, for $x \in \R^{d}$, which clearly epi-converges to $f(x) =  x^\top q x$. Then, it holds $f^*_n(z) = \nor{z}_{q_n}^2\in[0, \infty]$, and so by the above epi-continuity theorem of the convex conjugate,  $(f^*_n)_{n=1}^{\infty}$ epi-converges to $f^*(z) = \nor{z}_{q}^2$ and by  \eqref{eq:liminf} we obtain
$$  \nor{z}_{q}^2 \le \liminf_{n\to \infty} \nor{z_n}_{q_n}^2.$$
\end{remark}

\subsection{Neural networks}
An MLP consists of a sequence of hidden layers stacked between the input and output layers, where each node (i.e. component of the vector) in a layer is connected with all nodes in the subsequent layer. Let $L \ge 1$ denote the total number of layers, excluding the input one, let $d_{\din}$ denote the dimension of the input space,
$$ \bm{n} = (d_{\din} = n_0, n_1, \ldots, n_L=d_{\out})$$
denote the size of the layers, so that network output is a vector in $\R^{d_{\out}}$. The construction goes as follows: we fix a family of functions (called activation functions)
$$ \bm{\sigma} = (\sigma^{(1)}, \sigma^{(2)},\ldots, \sigma^{(L)} ) \quad \text{with} \quad \sigma^{(\ell)}: \R \to \R,$$
and two  sets of parameters (called, respectively, weights and biases),
$$\bm{W} = (W^{(1)}, W^{(2)}, \ldots, W^{(L)}), \quad  \bm{b} = (b^{(1)}, b^{(1)}, \ldots, b^{(L)}),$$
where, for every $\ell =1, \ldots, L$, one has $W^{(\ell)} \in \R^{n_{\ell} \times n_{\ell-1}}$ and $b^{(\ell)} \in \R^{n_{\ell}}$. We set $n_0:= d_{\din}$ the input dimension,
$$f^{(0)} : \R^{d_{\din}} \to \R^{n_0}, \quad f^{(0)}(x) = x,$$
and recursively define, for $\ell=1, \ldots, L$,
\begin{equation}\label{eq:fully-connected} f^{(\ell)} : \R^{d_{\din}} \to \R^{n_{\ell}}, \quad f^{(\ell)}(x) = W^{(\ell)}\sigma^{(\ell)} (f^{(\ell-1)}(x)) + b^{(\ell)},
 \end{equation}
where  $n_L := d_{\out}$ is the output dimension. The composition with $\sigma^{(\ell)}$ is always to be understood componentwise, i.e.,
$$ \sigma^{(\ell)} (f^{(\ell-1)}(x))_i := \sigma^{(\ell)}( f^{(\ell-1)}_i(x)).$$
Usually, one lets $\sigma^{(1)}(x) = x$, the identity function.  Moreover, we always consider activation functions  that are Lipschitz continuous, and popular choices for $\sigma^{(\ell)}$ include the ReLU function $x\mapsto \max\cur{x, 0}$, and the sigmoid (or logistic) function $x \mapsto (1+ e^{-x})^{-1}$. Furthermore, to keep annoying technicalities at bay, we often assume in the proofs that all activation functions are the same, i.e., $\sigma^{(\ell)} = \sigma$ for every $\ell=1, \ldots, L$ and  we only focus on the case of identically null biases, $b^{(\ell)} = 0$ for every $\ell$.
An MLP is naturally intended as a parametric family of functions $f(\cdot; \theta): \R^{d_{\din}} \to \R^{d_{\out}}$, where the parameters are given by the weights and biases. In particular, given an input $x \in \R^{d_{\din}}$, the output of the network is given by $f^{(L)}(x) \in \R^{d_{\out}}$. One can also think of a network as a sequence of parametrized feature maps, yielding kernels for each layer. A particular role in the derivation of the LDP is played by the \emph{post-activation} kernels

\begin{equation}\label{eq:v-def} v^{(\ell)}(x,y) := \frac 1 {n_\ell} \sum_{i=1}^{n_\ell} \sigma(f^{(\ell)}_i(x)) \sigma(f^{(\ell)}_i(y)),\quad\text{$\ell=1,\ldots,L.$}
\end{equation}
For a (finite) input set $\cX=\cur{x_i}_{i=1, \ldots, k} \subseteq \R^{d_{\din}}$, we introduce the notation  
$$ f^{(\ell)}(\cX) := \bra{ f^{(\ell)}_i(x_j)}_{i=1, \ldots, n_\ell; j=1, \ldots, k} \in \R^{n_\ell \times k},$$
and
$$ v^{(\ell)}(\cX) := \bra{ v^{(\ell)}(x_i, x_j)}_{i,j=1, \ldots,k} \in \R^{k \times k}_{\succeq 0}.$$

Neural network's weights are usually initialized at random, and a popular choice is given by the so-called LeCun initialization \cite{lecun2002efficient}, which consists in choosing weights and biases as independent Gaussian random variables, centered with variances proportional to the inverse of the size of the previous layer, i.e.,
\begin{equation}\label{eq:lecun} \EE\sqa{ (W^{(\ell)}_{i,j})^2} = \frac 1 {n_{\ell-1}} \quad \text{for every $\ell=1, \ldots, L$ and every $i$, $j$.} \end{equation}

\begin{remark}
Biases in LeCun initialization are instead usually taken centered with unit variance, i.e.,
$$\EE\sqa{ (b^{(\ell)}_{i})^2} = 1 \quad \text{for every $\ell=1, \ldots, L$ and every $i$,}$$
but for the sake of simplicity, we will assume instead in this work that biases are identically null.
\end{remark}

\begin{remark}\label{rem:conditional-gaussian}
A key structural property of MLPs with Gaussian weights is the
following conditional Gaussianity. Fix a finite input set $\cX=\{x_i\}_{i=1, \ldots, k}$ and
a layer $\ell\ge1$. Then, conditionally on the previous layer outputs
$f^{(\ell-1)}(\cX)$, the vectors
\[
\bigl(f^{(\ell)}_i(\cX)\bigr)_{i=1}^{n_\ell}:=\bigl((f^{(\ell)}_i(x_1),\ldots,f^{(\ell)}_i(x_k))\bigr)_{i=1}^{n_\ell}
\]
are i.i.d.\ centered Gaussian random vectors in $\R^k$, each with covariance matrix
$v^{(\ell-1)}(\cX)$. This follows from the fact that \eqref{eq:fully-connected} defines a linear transformation of the weights $W^{(\ell)}$, with (conditionally) constant coefficients. In particular, the conditional law of
$\bigl(f^{(\ell)}_i(\cX)\bigr)_{i=1}^{n_\ell}$ given $f^{(\ell-1)}(\cX)$ depends on the
previous layer only through the post-activation kernel $v^{(\ell-1)}(\cX)$. As a consequence, given $f^{(\ell-1)}(\cX)$, one has
\[
v^{(\ell)}(\cX)
\;\stackrel{\mathrm{law}}{=}\;
\frac{1}{n_\ell}\sum_{i=1}^{n_\ell}
\sigma\!\left(\sqrt{v^{(\ell-1)}(\cX)}\,N_i\right)^{\otimes2},
\]
where $(N_i)_{i\ge1}$ are i.i.d.\ standard Gaussian vectors in $\R^k$, independent of
$v^{(\ell-1)}(\cX)$. This observation is the starting point for the inductive large
deviation analysis of the kernel sequence, see the proof of Proposition \ref{prop:ldp-kernel}.
\end{remark}

\subsection{Large deviations}

We recall some basic facts about the theory of large deviations, referring to monographs such as \cite{dembo2009large, rassoul2015course} for details. A sequence of random variables $(X_n)_{n=1}^\infty$ taking values in a Polish space $\mathsf X$ is said to satisfy the large deviation principle (LDP) with speed $(a_n)_{n=1}^\infty \subseteq (0, \infty)$ and rate function $I: \mathsf X \to [0, \infty]$ if $a_n \uparrow \infty$,  $I$ has compact sub-level sets \footnote{Throughout this paper we choose to adopt a slightly simplified terminology, since a rate function with compact sub-level sets is usually called good rate function in standard textbooks on large deviations theory.}, and for every Borel set $B \subseteq \mathsf X$ it holds
$$ - \inf_{x \in B^\circ} I(x) \le \liminf_{n \to \infty} \frac 1 {a_n} \log P(X_n \in B) \le \limsup_{n \to \infty} \frac 1 {a_n} \log P(X_n \in B) \le - \inf_{x \in \ov B} I(x).$$
The above formulation is reminiscent of the classical portmanteau theorem of convergence in law, and indeed large deviations provide a natural extension of such convergence to rare events, i.e., events whose probability decays exponentially fast in $n$. Indeed, often writes the above inequalities informally as
$$ P(X_n \in B) \approx \exp\bra{- a_n \inf_{x \in B} I(x) }.$$
\begin{remark}\label{re:22012026}
Note that from the definition of LDP, if $(X_n)_{n=1}^{\infty}$ satisfies the LDP with speed $(a_n)_{n=1}^\infty$ and rate function $I$, and for another sequence $(a'_n)_{n=1}^\infty$ it holds $\lim_{n}a_n/a_n' = \beta\in (0, \infty)$, then $(X_n)_{n=1}^\infty$ satisfies the LDP with speed $(a'_n)_{n=1}^\infty$ and rate function $x\mapsto \beta I(x)$; actually, this holds true also if  $\beta=\infty$, provided that $I$ has a unique zero, as shown in \cite[Lemma 1]{giulianomaccipacchiarottiJAP2024}. This will be precisely used in Proposition \ref{prop:ldp-kernel}. 

Furthermore, if $X_n=x_0$ is a constant and deterministic sequence, then the LDP trivially holds with any speed and rate function $I(x) = \infty 1_{\cur{x\neq x_0}}$, i.e.,\ it equals $+\infty$ if $x \neq x_0$ and $0$ if $x=x_0$.
\end{remark}

Hereafter, we collect some results from the general theory of large deviations that will be used throughout this work.
We begin with stating the contraction principle, see \cite[Theorem 4.2.1]{dembo2009large} for a proof.

\begin{proposition}[Contraction principle]
Let $(X_n)_{n=1}^\infty$ be a sequence of random variables with values in a Polish space $\mathsf X$ satisfying the LDP with speed $(a_n)_{n=1}^\infty$ and rate function $I: \mathsf X \to [0, \infty]$, and let $F: \mathsf X \to \mathsf Y$ be a continuous map to another Polish space $\mathsf Y$. Then, the sequence $(F(X_n))_{n=1}^\infty$ satisfies the LDP with speed $(a_n)_{n=1}^\infty$ and rate function
$$ J(y) := \inf_{x \in F^{-1}(\cur{y})} I(x), \quad \text{for $y \in \mathsf Y$.}$$
\end{proposition}

Next, we recall Cramér's theorem, see \cite[Corollary 6.16]{dembo2009large} or \cite[Section 4.3]{rassoul2015course} for proofs. Let $X$ be a random variable with values in $\mathbb R^d$. Hereafter, we denote by $M(\lambda) := \EE\sqa{ \exp\bra{ \lambda^\top X }}$, $\lambda\in\mathbb R^d$, the moment generating function (MGF) of $X$.

\begin{theorem}[Cramér's theorem]
Let $(X_n)_{n=1}^\infty$ be i.i.d.\ random variables with values in $\R^d$ and set $\bar{X}_n = \frac 1 n \sum_{i=1}^nX_i$. If $M(\lambda)$ is finite in a neighborhood of $\lambda = 0$, then $\bar{X}_n$ satisfies the LDP with speed $n$ and rate function
$$ I(x) = \sup_{\lambda \in \R^d} \bra{ \lambda^\top x - \log M(\lambda) }.$$
\end{theorem}

In some situations it is useful to replace a sequence of random variables with another one that is ``exponentially close'' to the original one. For our purpose, the following result will suffice and we refer to \cite[Section 4.2.2]{dembo2009large} for a more general presentation.

\begin{proposition}[Exponential equivalence]\label{prop:exp-equivalence}
Let $(X_n)_{n=1}^\infty$ and $(X_n')_{n=1}^\infty$ be two sequences of random variables with values in $\R^d$. If $(X_n)_{n=1}^\infty$ satisfies the LDP with speed $(a_n)_{n=1}^\infty$ and rate function $I$, and it holds
\begin{equation}\label{eq:condition-exp-equiv}
 \lim_{n \to \infty} \frac 1 {a_n} \log P( |X_n - X_n'| > \delta ) = -\infty, \quad \text{for every $\delta>0$,}
\end{equation}
then $(X_n')_{n=1}^\infty$ satisfies the LDP with same speed $(a_n)_{n=1}^\infty$ and same rate function $I$.
\end{proposition}

Condition \eqref{eq:condition-exp-equiv} can be  checked by suitable exponential moment bounds. Precisely, if
\begin{equation}\label{eq:condition-exp-equiv-moment}
 \lim_{n \to \infty} \frac 1 {a_n} \log \EE\sqa{ \exp\bra{a_n \alpha  |X_n - X_n'| }}  = 0, \quad \text{for every $\alpha>0$,}
\end{equation}
then \eqref{eq:condition-exp-equiv} holds. Indeed, for any $\delta>0$ and $\alpha>0$,  we bound from above by Markov's inequality
\begin{equation*}\begin{split}
\frac{1}{a_n} \log P( |X_n - X_n'| > \delta ) & \le \frac{1}{a_n} \log   \EE\sqa{ \exp\bra{ a_n \alpha \bra{|X_n - X_n'|-  \delta } }}\\
& = \frac{1}{a_n} \log   \EE\sqa{ \exp\bra{ a_n \alpha |X_n - X_n'|}  } -  \alpha \delta.
\end{split}\end{equation*}
Letting first $n \to \infty$ and using \eqref{eq:condition-exp-equiv-moment}, we find
$$ \limsup_{n\to \infty} \frac{1}{a_n} \log P\bra{|X_n - X_n'| > \delta}  \le -\alpha \delta. $$
Letting then $\alpha \to \infty$, we conclude that \eqref{eq:condition-exp-equiv} holds true.

Our approach to prove LDPs for MLPs is based on a result due to Chaganty \cite{chaganty1997large}. Given two sequences of dependent random variables, it provides sufficient conditions which guarantee that the sequence of joint laws satisfies an LDP. To keep the exposition short, we state it in a slightly simplified form.

Let $(X_n)_{n=1}^\infty$ and $(Z_n)_{n=1}^\infty$ be two sequences of random variables with values on the Polish spaces $\mathsf X$ and $\mathsf Z$,  respectively. We say that the sequence of their regular conditional probabilities
$\bra{ P( X_n \in \cdot | Z_n = \cdot)}_{n=1}^\infty$ 
satisfies the LDP \emph{continuity condition}
with speed $(a_n)_{n=1}^\infty$ and function  $J: \mathsf X\times\mathsf Z \to [0,\infty]$, $(x,z) \mapsto J(x|z)$, 
if
\begin{enumerate}
\item $J(\cdot|\cdot)$ is jointly lower semicontinuous and, for every $L< \infty$ and $\mathsf K \subseteq \mathsf Z$ compact, the set
\begin{equation}\label{eq:compact-sublevel} \bigcup_{z \in\mathsf K} \cur{ x \in \mathsf X: J(x|z) \le L}\end{equation}
is compact;
\item for every $z \in \mathsf Z$ and $(z_n)_{n=1}^\infty \subseteq \mathsf Z$ such that $z_n \to z$, the sequence $(X_n|Z_n = z_n)_{n=1}^\infty$ satisfies the LDP with speed $(a_n)_{n=1}^\infty$ and rate function $J(\cdot| z)$, i.e., for every Borel set $B \subseteq \mathsf X$, 
\begin{equation}\label{eq:definition-ldp-conditional}\begin{split}
 - & \inf_{x \in B^\circ} J(x|z) \le \liminf_{n \to \infty} \frac 1 {a_n} \log P(X_n \in B | Z_n = z_n)\\  & \quad \phantom{ffffffffffff}\le \limsup_{n \to \infty} \frac 1 {a_n} \log P(X_n \in B|Z_n=z_n) \le - \inf_{x \in \ov B} J(x|z).
 \end{split}
\end{equation} 
\end{enumerate}

\begin{remark}\label{re2:22012026}
We notice that, if $J$ is jointly lower semciontinuous, then the set \eqref{eq:compact-sublevel} is always closed. Indeed, given a sequence $(x_n)_{n=1}^\infty$ in the set converging to some $x$, then there exists $(z_n)_{n=1}^\infty\subseteq\mathsf K$ such that $J(x_n|z_n) \le L$ for every $n$. Up to taking a subsequence, since $\mathsf K$ is compact, we have $z =\lim_{n}z_n \in \mathsf K$ and therefore by lower semicontinuity of $J$ it holds $J(x|z) \le \liminf_{n} J(x_n|z_n) \le L$. Hence, to establish (1), it is sufficient to argue that $J$ is jointly lower semicontinuous and the set \eqref{eq:compact-sublevel} is pre-compact.
\end{remark}

The following result is a combination of \cite[Theorem 2.3, Lemma 2.6]{chaganty1997large}.

\begin{lemma}[Chaganty's lemma]\label{lem:chaganty}
Let $((X_n, Z_n))_{n=1}^\infty$ be random variables with values in $\mathsf X \times \mathsf Z$, where $\mathsf X$ and $\mathsf Z$ are Polish spaces. If the sequence $(Z_n)_{n=1}^\infty$ satisfies the LDP with speed $(a_n)_{n=1}^\infty$ and rate function $I_{Z}: \mathsf Z \to [0, \infty]$, and the sequence 
$\bra{ P( X_n \in \cdot | Z_n = \cdot)}_{n=1}^\infty$ 
satisfies the LDP continuity condition, with speed $(a_n)_{n=1}^\infty$ and function $J: \mathsf X \times \mathsf Z \to [0,\infty]$, then the sequence $((X_n, Z_n))_{n=1}^\infty$ satisfies the LDP, with speed $(a_n)_{n=1}^\infty$ and rate function
$$ I(x,z) := J(x|z) + I_{Z}(z) , \quad \text{for $(x,z) \in \mathsf X \times \mathsf Z$}.$$
\end{lemma}

\begin{example}\label{ex:gaussian-chaganty}
Consider a sequence $((X_n, V_n))_{n}$ where $(V_n)_n \in \R^{k \times k}_{\succeq 0}$ satisfies the LDP with speed $n$ and rate function $I_V$ and 
$$ X_n = \frac{\sqrt{V_n}}{\sqrt{n}} N,$$
where $N$ is a standard Gaussian vector in $\R^{k}$, independent of $(V_n)_n$. Then, by standard Gaussian computations, the conditional law of $X_n$ given $V_n = v_n$ is Gaussian with mean zero and covariance matrix $v_n/n$,
and the sequence
$\bra{ P( X_n \in \cdot | V_n = \cdot)}_{n=1}^\infty$ 
satisfies the LDP continuity condition with speed $n$ and $J(x|v) := \frac 1 2 \nor{x}_v^2$,  defined 
in the extended sense (see Section \ref{sec:gennot}) if $v$ is not invertible.  Indeed,  joint lower semicontinuity of $J$ is discussed in Remark~\ref{rem:zqz-lsc} and compactness of \eqref{eq:compact-sublevel} is elementary, thus establishing condition (1), while condition (2) follows either by a direct application of G\"artner-Ellis theorem 
\cite[Theorem 2.3.6]{dembo2009large} or, using only the results introduced in this section, by first applying Cram\'er's theorem for a fixed $v$ (and representing the Gaussian $\sqrt{n} N$ as a sum of i.i.d.\ standard Gaussian variables) and then arguing exponential equivalence of a sequence $X_n = \sqrt{v_n/n}N$ with $X'_n:= \sqrt{v/n}N$, whenever $v_n \to v$. 

 It follows from Chaganty's lemma that the sequence $((X_n, V_n))_{n=1}^{\infty}$ satisfies the LDP with speed $n$ and rate function
$$ I(x,v) = \frac 1 2 \nor{x}_v^2 + I_V(v).$$
By the contraction principle we then have that the sequence of marginals $(X_n)_{n=1}^{\infty}$ satisfies the LDP with the same speed $n$ and rate function
$$ I_X(x) := \inf_{v \in \R^{k \times k}_{\succeq 0}} \Big\{\frac 1 2\nor{x}_v^2 + I_V(v)\Big\} .$$
\end{example}

Next, we recall that large deviations analogue of tightness for probability measures is exponential tightness. A sequence of random variables $(X_n)_{n=1}^\infty$ taking values in a Polish space $\mathsf X$ is said to be exponentially tight with speed $(a_n)_{n=1}^\infty \subseteq (0, \infty)$ if for every $M>0$ there exists a compact set $\mathsf K = \mathsf{K}(M) \subseteq \mathsf X$ such that
\begin{equation}\label{eq:exp-tight-def}
 \limsup_{n \to \infty} \frac {1}{a_n} \log P( X_n \notin \mathsf K) \le -M.
\end{equation}
Such a condition follows easily if one provides a function $\psi: \mathsf X \to [0, \infty]$ with compact sub-level sets such that
$$ \limsup_{n \to \infty} \frac 1 {a_n} \log \EE\sqa{ \psi(X_n)^{a_n}  } < \infty.$$
Indeed, by Markov's inequality, for every $b>0$,
$$ P( \psi(X_n) > b ) = P(\psi(X_n)^{a_n} > b^{a_n}) \le  \frac{ \EE\sqa{ \psi(X_n)^{a_n}}}{ b^{a_n}},  $$
hence, setting $\mathsf K_b := \cur{x \in \mathsf X\, : \psi(x) \le b}$, we have
$$\limsup_{n \to \infty} \frac {1}{a_n} \log P( X_n \notin \mathsf{K}_b) \le \limsup_{n \to \infty} \frac 1 {a_n} \log \EE\sqa{ \psi(X_n)^{a_n}  } - b $$
and therefore, given $M>0$, one can find  $b$ sufficiently large so that \eqref{eq:exp-tight-def} holds with $\mathsf K = \mathsf{K}_b$.

In the case of processes with continuous trajectories, to establish exponential tightness one can rely on 
a variant of a result due
to Schied \cite{schied1994criteria, , schied1997moderate, djehiche1998large}. 
We provide a detailed statement here, together with a proof in Appendix~\ref{app:schied}.

\begin{theorem}[Exponential tightness criterion]
\label{thm:schied}
Let $Z_n=(Z_n(x))_{x\in [0,1]^d}$, $n\geq 1$, be a sequence of stochastic processes taking values in $\R^k$, with continuous sample paths. 

The sequence $(Z_n)_{n =1}^\infty$ is exponentially tight in $C([0,1]^d,\mathbb R^k)$,  endowed with the topology of uniform convergence, with speed $n$, if there exist $\bar{n} \in \mathbb{N}$,  $\alpha>0$, $\cost < \infty$ such that, for all $x,y \in [0,1]^d$ with $x \neq y$,
\begin{equation}\label{eq:schied-increments}
\sup_{n\ge \bar{n}} \frac {1}{n} \log \EE\sqa{ \bra{1+ \frac{|Z_n(x)- Z_n(y)| }{|x-y|^{\alpha} }}^{n}} \le \cost 
\end{equation}
and 
\begin{equation}\label{eq:schied-zero} \sup_{n \ge \bar{n}} \frac {1}{n} \log \EE\sqa{ \bra{1+ \abs{Z_n(0)} }^{n} }\le \cost. \end{equation}
\end{theorem}

\begin{remark}\label{rem:squared-schied}
By the elementary inequality $1+a \le \sqrt{2  (1+a^2) }$ and Cauchy-Schwarz inequality, one can replace in the hypothesis both $|Z_n(x)- Z_n(y)| /|x-y|^{\alpha}$ in  \eqref{eq:schied-increments} and $\abs{Z_n(0)}$ in \eqref{eq:schied-zero} with their squared values, without affecting the thesis.
\end{remark}

We end this section with a result that allows to represent the rate function associated to a sequence of processes $(Z_n)_{n=1}^\infty$ taking values in $C(\mathsf K, \R^k)$, where $\mathsf K\subset\mathsf X$ is a compact set of a Polish space $\mathsf X$, in terms of the rate functions of its finite-dimensional marginals. The result is a straightforward consequence of the classical Dawson-Gärtner theorem for projective limits and the exponential tightness, see \cite[Theorem 4.6.1 and Theorem 4.2.4]{dembo2009large}. Hereafter, for a finite set $\mathsf F$ we denote by $|\mathsf F |$ its cardinality.

\begin{proposition}[LDP for processes]\label{prop:dawson}
Let $\mathsf X$ be a Polish space, $\mathsf K\subset\mathsf X$ a compact set, $Z_n=(Z_n(x))_{x\in K}$, $n\geq 1$, a sequence of stochastic processes taking values in $\R^k$, and let $(a_n)_{n=1}^\infty$ be a sequence of positive real numbers with $a_n \uparrow \infty$. Assume that for every finite set $\mathsf F \subseteq \mathsf K$, the sequence $(Z_n(x))_{x \in \mathsf F}$, $n\geq 1$, satisfies the LDP with speed $(a_n)_{n=1}^\infty$ and rate function $I_{\mathsf F}: \R^{k \times |\mathsf F|} \to [0, \infty]$ and that the sequence $(Z_n)_{n=1}^\infty$ is exponentially tight in $C(\mathsf K; \R^k)$ with speed $(a_n)_{n=1}^\infty$. Then, the sequence $(Z_n)_{n=1}^\infty$ satisfies the LDP in $C(\mathsf K; \R^k)$ with speed $(a_n)_{n=1}^\infty$ and rate function
$$ I(z) := \sup_{\mathsf F \subseteq \mathsf K,\,\, \text{finite} } I_{\mathsf F}( (z(x))_{x \in \mathsf F} ), \quad \text{for $z \in C(\mathsf K, \R^k)$.}$$
\end{proposition}

\section{Exponential tightness}\label{sec:exp-tight}

In this section, we establish exponential tightness for MLPs with Gaussian weights and Lipschitz activation functions.

Before we state and prove the main result of this section, we provide a bound on moments of $\chi^2$-distributed random variables.

\begin{lemma}\label{lem:gamma}
Let $(K_n)_{n=1}^\infty$ be a sequence of random variables with $K_n \stackrel{law}{=} \chi^2_n$ for every $n \ge 1$. Then
$$ \sup_{n \ge 1} \frac 1 n \log \EE\sqa{ \bra{1 + \frac{K_n}{n} }^{\gamma n} } < \infty, \quad \text{ for every $\gamma>0$.} $$
\end{lemma}

\begin{proof}
We split the expectation over the event $A = \cur{ K_n/ n\le 1}$ and its complement, so that by triangle inequality for the $L^n$-norm, letting $1_A$ denote the indicator function of a set $A$, we have
\begin{eqnarray*} 
 \EE\sqa{ \bra{1 + \frac{K_n}{n} }^{\gamma n}}^{1/n} & = &\nor{\bra{1 + \frac{K_n}{n } }^{\gamma} }_{L^n} \\
& \le& \nor{1_A \bra{1 + \frac{K_n}{n} }^\gamma }_{L^n} + \nor{1_{A^c} \bra{1 + \frac{K_n}{n} }^\gamma }_{L^n} \\
& \le& 2^\gamma\bra{1+ \nor{1_{A^c} \bra{ \frac{K_n}{n} }^\gamma }_{L^n}}.
\end{eqnarray*}
Next, letting ${\rm Gamma}(\alpha,\theta)$ denote the gamma distribution with parameters $\alpha,\theta>0$, we use that $\chi^2_n = {\rm Gamma}(n/2,1/2)$, so that $K_n/n$ has law ${\rm Gamma}(n/2,n/2)$ and therefore, recalling the formula for its moments, we have
$$  \EE\sqa{ 1_{A^c} \bra{ \frac{K_n}{n} }^{n\gamma} } \le \EE\sqa{ \bra{ \frac{K_n}{n} }^{n\gamma} } = \bra{\frac{2}{n}}^{\gamma n} \frac{\Gamma(n(\gamma+1/2))}{\Gamma(n/2)}, $$
where $\Gamma(\cdot)$ denotes the Euler gamma function.
By Stirling's approximation for the gamma function we have
$$ \Gamma(x) = \bra{\frac{x}{e}} ^x  \sqrt{\frac{2\pi}{x}} \bra{1 + O\bra{\frac 1 x }}, \quad \text{as $x \to \infty$.}$$
Therefore, for some positive constant $\cost = \cost(\gamma)<\infty$,
\begin{eqnarray*}  \bra{\frac{2}{n}}^{\gamma n} \frac{\Gamma(n(\gamma+1/2))}{\Gamma(n/2)} & \le& \cost \bra{\frac{2}{n}}^{\gamma n} \frac{ (n(\gamma+1/2)/e)^{n(\gamma+1/2)} }{ (n/(2e))^{n/2}  } \\
& \le& \cost^n  \frac{n^{(\gamma+1/2)n} }{n^{\gamma n} n^{n/2} } \le \cost^n,\quad\text{for all $n$ large enough,}
\end{eqnarray*}
which yields the claim.
\end{proof}

Next, assuming that the activation functions are Lipschitz continuous, we establish exponential tightness for the sequence of network outputs. 

\begin{proposition}
Fix $L \ge 1$, input $d_{\din}$ and output  $d_{\out}\ge 1$ dimensions, Lipschitz activation functions $(\sigma^{(\ell)})_{\ell=0, 1,\ldots, L}$, and consider a sequence of MLPs outputs $
f^{(L)}_{(n)}: \R^{d_{\din}} \to \R^{d_{\out}}$ with Gaussian weights as in \eqref{eq:lecun} and hidden layers widths $n_\ell = n_\ell(n)$ such that
$$   \liminf_{n \to \infty} \frac{ n_{\ell}}{n} >0 \quad \text{for every $\ell=1, \ldots, L-1$.} $$
Then, for every compact $\cX \subseteq \R^{d_{\din}}$, the sequence of real-valued continuous processes $\big(\frac{1}{\sqrt{n}} f^{(L)}_{(n)}\big)_{n=1}^{\infty}$, defined on $\cX$, is exponentially tight on $C(\cX; \R^{d_{\out} })$ (equipped with the uniform topology) with speed $n$.
\end{proposition}

\begin{proof}
To simplify the notation, we assume that $\sigma^{(\ell)} = \sigma$ for each $\ell$. Moreover, since up to translations and dilations one has $\cX \subseteq [0,1]^{d_{\din}}$, we directly prove the result in the case $\cX = [0,1]^{d_{\din} }$. 

The assumption entails that, for some $\beta>0$ and some $\bar{n} \ge 1$ it holds $n_\ell(n)\ge n \beta$ for every $n \ge \bar{n}$. To avoid treating the output layer (whose width is fixed) as a separate case, we artificially extend the output dimension of the network so that $f^{(L)}_{(n)}$ takes values in $\R^{n}$, i.e., we let $n_L(n) := n$ and without loss of generality, we also assume that $\bar{n} \ge d_{\out}$, so that the original network output, taking values in $\R^{d_{\out}}$,  can be obtained by considering the first $d_{\out}$ coordinates of the extended one. Let us notice that such extension is always possible up to enlarging the probability space where $f^{(L)}_{(n)}$ is defined to accomodate further i.i.d.\ weights for the output layer.

With this convention, we start proving, by induction over $\ell=0, 1, \ldots, L$ that, for every $\gamma>0$, there exists $\cost = \cost(\beta, \bar{n}, \gamma, {\bf \sigma}, \cX) < \infty$ such that
\begin{equation}\label{eq:induction-tightness-increments} \sup_{n \ge \bar{n} } \frac 1 n \log \EE\sqa{ \bra{1 +   \frac 1 {n_\ell} \sum_{i=1}^{n_\ell}  \frac{ \abs{f^{(\ell)}_{i}(x) - f^{(\ell)}_{i}(x') }^2 } {|x-x'|^{2} } }^{\gamma n } } \le \cost \quad \text{for every $x, x' \in \cX$}.\end{equation}
Here and below, to keep notation simple, we write $f^{(\ell)}_{i}$ instead of $f^{(\ell)}_{(n), i}$. In the case $\ell=0$ the claim is trivial because $f^{(0)}(x)=x$. Assuming that the statement holds for $\ell-1$, we prove it for $\ell$. By Remark~\ref{rem:conditional-gaussian} we can write, conditionally upon  $( f^{(\ell-1)}(x), f^{(\ell-1)}(x'))$, for every $i=1,\ldots, n_{\ell}$,
$$
f^{(\ell)}_i(x) - f^{(\ell)}_i(x')  = \sum_{j=1}^{n_{\ell-1}} W^{(\ell)}_{i,j} \bra{ \sigma(f^{(\ell-1)}_j(x)) -  \sigma(f^{(\ell-1)}_j(x')) }  
 \stackrel{law}{=} \sqrt{v^{(\ell)}_n}  N_i, 
$$
where we set
$$ v^{(\ell)}_n :=   \frac {1} {n_{\ell-1}} \sum_{j=1}^{n_{\ell-1}} \abs{ \sigma(f^{(\ell-1)}_j(x)) -  \sigma(f^{(\ell-1)}_j(x')) }^2 $$
and $(N_i)_{i=1}^{n_\ell}$ are i.i.d.\ standard (real) Gaussian random variables, independent of $f^{(\ell-1)}(x)$ and $f^{(\ell-1)}(x')$.  Therefore, still  conditionally upon  $( f^{(\ell-1)}(x), f^{(\ell-1)}(x'))$, we have
$$ \frac 1 {n_\ell} \sum_{i=1}^{n_\ell}  \frac{ \abs{f^{(\ell)}_i(x) - f^{(\ell)}_i(x') }^2 } {|x-x'|^{2} } \stackrel{law}{=} \frac{v^{(\ell)}_n}{|x-x'|^ 2} \cdot \frac{K_{n_{\ell}}}{n_\ell}, $$
where $K_{n_\ell}$ is $\chi^2_{n_\ell}$-distributed. Using the elementary inequality $1+ab \le (1+a)(1+b)$, for $a, b \ge 0$, we find therefore  that
$$ \EE\sqa{ \bra{1+ \frac 1 {n_\ell} \sum_{i=1}^{n_\ell}  \frac{ \abs{f^{(\ell)}_i(x) - f^{(\ell)}_i(x') }^2 } {|x-x'|^{2} }}^{\gamma n}  } \le \EE\sqa{ \bra{ 1+\frac{v^{(\ell)}_n}{|x-x'|^ 2}}^{\gamma n} \bra{1+\frac{K_{n_{\ell}}}{n_\ell }}^{\gamma n } }.$$
By independence
we have
$$ \EE\sqa{ \bra{ 1+\frac{v^{(\ell)}_n}{|x-x'|^ 2}}^{\gamma n} \bra{1+\frac{K_{n_{\ell}}}{n_\ell}}^{\gamma n } } = \EE\sqa{ \bra{ 1+\frac{v^{(\ell)}_n}{|x-x'|^ 2}}^{\gamma n}} \EE\sqa {\bra{1+\frac{K_{n_{\ell}}}{n_\ell }}^{\gamma n } }.$$
The second term is readily estimated from above first by H\"older's inequality with exponent $n_{\ell}/(n\beta)\ge 1$,
$$ \EE\sqa {\bra{1+\frac{K_{n_{\ell}}}{n_\ell }}^{\gamma n } } \le \EE\sqa {\bra{1+\frac{K_{n_{\ell}}}{n_\ell }}^{\gamma n_\ell/\beta } }^{n\beta/n_{\ell} } $$
and then using Lemma~\ref{lem:gamma}, which, for some constant $\cost = \cost(\gamma, \beta) < \infty$, yields
\begin{equation*}
\frac {1} {n} \log \EE\sqa {\bra{1+\frac{K_{n_{\ell}}}{n_\ell }}^{\gamma n_\ell/\beta } }^{n\beta/n_\ell}   = \frac{\beta}{n_\ell}  \log \EE\sqa {\bra{1+ \frac{K_{n_{\ell}}}{n_{\ell} }}^{\gamma n_\ell/\beta } } \le \cost.
\end{equation*}

Hence, to conclude with \eqref{eq:induction-tightness-increments} it is sufficient to bound from above the term
$$ \frac 1 n \log \EE\sqa{ \bra{ 1+\frac{v^{(\ell)}_n}{|x-x'|^ 2}}^{\gamma n} }.$$
Since the activation function is Lipschitz continuous, we find
\begin{equation*}\begin{split} v^{(\ell)}_n & =  \frac {1} {n_{\ell-1}} \sum_{j=1}^{n_{\ell-1}} \abs{ \sigma(f^{(\ell-1)}_j(x)) -  \sigma(f^{(\ell-1)}_j(x')) }^2 \\ & \le \Lip\bra{ \sigma}^2\frac{1 }{n_{\ell-1}}\sum_{j=1}^{n_{\ell-1}} \abs{ f^{(\ell-1)}_j(x) -  f^{(\ell-1)}_j(x') }^2
\end{split}\end{equation*}
and therefore, again using $1+ab\le (1+a)(1+b)$,
$$ 1+\frac{v^{(\ell)}_n}{|x-x'|^ 2} \le \bra{1+  \Lip\bra{ \sigma}^2 }\bra{1+\frac 1 {n_{\ell-1}} \sum_{i=1}^{n_{\ell-1}} \frac{ \abs{f_i^{(\ell-1)}(x) -f_i^{(\ell-1)}(x') }^2 } {|x-x'|^ 2} }.$$
We conclude that
\begin{equation*}\begin{split}  \frac 1 n \log  \EE\sqa{ \bra{ 1+\frac{v^{(\ell)}_n}{|x-x'|^ 2}}^{\gamma n}} & \le  \gamma \log\bra{1+\beta \Lip\bra{ \sigma}^2} \\
& + \frac 1 n  \log  \EE\sqa{ \bra{ 1+ \frac 1 {n_{\ell-1}} \sum_{i=1}^{n_{\ell-1}} \frac{ \abs{f_i^{(\ell-1)}(x) -f_i^{(\ell-1)}(x') }^2 } {|x-x'|^ 2}}^{\gamma  n}}
\end{split}\end{equation*}
and therefore the validity of \eqref{eq:induction-tightness-increments} follows, taking into account the inductive assumption.

Arguing similarly -- actually using only the fact that activation functions have linear growth -- we can prove that for any given $x_0 \in \cX$ there exists $\cost = \cost(\beta, \bar{n}, \gamma, \sigma, x_0)$ such that
\begin{equation}\label{eq:induction-tightness-point} \sup_{n \ge \bar n } \frac 1 n \log \EE\sqa{ \bra{1 +   \frac 1 {n_\ell} \sum_{i=1}^{n_\ell}  \abs{f^{(\ell)}_i(x_0)}^2 } ^{\gamma n }}\leq\cost\end{equation}
for $\ell=0,1, \ldots, L$.

We are now in a position to apply Theorem~\ref{thm:schied} to $Z_n := \frac{1}{\sqrt{n}} f^{(L)}$  on the state space $\R^{d_{\out}}$, also taking into account Remark~\ref{rem:squared-schied}. Indeed, the assumptions are satisfied: condition \eqref{eq:schied-increments}
is implied by \eqref{eq:induction-tightness-increments} in the case $\ell=L$, recalling that we set $n_L(n)= n$ and $\bar n \ge d_{\out}$, so that
\begin{equation*}\begin{split} &  \sup_{n \ge \bar{n} } \frac 1 n \log \EE\sqa{ \bra{1 +   \sum_{i=1}^{d_{\out } }  \frac{ \abs{ \frac 1 {\sqrt{n} } f^{(L)}_{i}(x) - \frac{1}{\sqrt{n}} f^{(L)}_{i}(x') }^2 } {|x-x'|^{2} } }^{\gamma n } } \le\phantom{gggggggggggg}\\
&  \phantom{gggggllllllggggggg}\sup_{n \ge \bar{n} } \frac 1 n \log \EE\sqa{ \bra{1 +   \frac 1 {n} \sum_{i=1}^{n }  \frac{ \abs{f^{(L)}_{i}(x) - f^{(L)}_{i}(x') }^2 } {|x-x'|^{2} } }^{\gamma n } } \le \cost;
\end{split}\end{equation*}
condition \eqref{eq:schied-zero}
is obtained similarly from \eqref{eq:induction-tightness-point} in the case $\ell=L$.
\end{proof}

\section{Large deviations}\label{sec:ldp-finite}

Given a Lipschitz continuous (activation) function $\sigma: \R \to \R$ and a positive semidefinite matrix $q \in \R^{k \times k}_{\succeq 0}$, we define the conditional $\log$-MGF, for $\lambda \in  \R^{k \times k}_{\sym}$,
$$\Lambda^{\sigma}(\lambda|q) = \log \EE\sqa{ \exp\bra{ \tr\bra{\lambda \sigma(\sqrt{q} N)^{\otimes 2 } }}} \in (-\infty, \infty],$$
where $N$ denotes a standard Gaussian vector with values in $\R^k$, and we refer to Section~\ref{sec:notation} for  the general notation. Taking the (partial) Legendre-Fenchel transform with respect to $\lambda$ of $\Lambda^{\sigma}$, we define the function $$J^{\sigma}(v | q) := \sup_{\lambda } \cur{ \tr\bra{\lambda v} - \Lambda^{\sigma}(\lambda|q)} \in [0, \infty],\quad v\in\R^{k \times k}_{\succeq 0}. $$
Note that $J^\sigma(\cdot|q)$ has a unique zero.

The main technical result of this section is the following lemma.

\begin{lemma}\label{lem:cramer-best}
If $\sigma: \R \to \R$ is Lipschitz continuous, then:
\begin{enumerate}
\item $J^\sigma(\cdot|\cdot)$ is jointly lower semicontinuous and, for every $L<\infty$ and $\mathsf K \subseteq \R^{k \times k}_{\succeq 0}$ compact, the set
\begin{equation*}
\bigcup_{q \in \mathsf K} \cur{v\in\R^{k \times k}_{\succeq 0}: J^\sigma(v|q) \le L}\end{equation*}
is compact.
\item Let $(N_i)_{i=1}^\infty$ be i.i.d.\ standard Gaussian random variables with values in $\R^k$. Then, for any sequence $(q_n)_n \subseteq \R^{k \times k}_{\succeq 0}$  converging to $q \in \R^{k \times k}_{\succeq 0}$, the process $(V_n)_{n=1}^{\infty}$, with
$$V_n := \frac 1 n \sum_{i=1}^n \sigma( \sqrt{q_n} N_i)^{\otimes 2},$$
satisfies the large deviation principle in $\R^{k \times k}$ with speed $n$ and rate function $J^{\sigma}(\cdot | q)$. 
\end{enumerate}
\end{lemma}

\begin{proof}

\newcounter{step}
\setcounter{step}{0}
We split the proof into multiple steps.

\textbf{Step \thestep: basic inequalities.} \stepcounter{step}
 Using that $\sigma$ is Lipschitz continuous,  for every $\eta \in \R^k$, $q, q' \in \R^{k \times k}_{\succeq 0}$, it holds
$$\abs{ \sigma( \sqrt{q} \eta ) }   \le \nor{\sigma}_{\Lip} \bra{1 + \abs{ \sqrt{q} \eta} } \le \nor{\sigma}_{\Lip} \bra{1+ \nor{\sqrt{ q }}_{\op}\abs{\eta} }
 \le  \nor{\sigma}_{\Lip} \bra{1+ \sqrt{\nor{ q  }_{\op} }} \bra{1+\abs{\eta}}$$
and
\begin{eqnarray*}\abs{  \sigma( \sqrt{q} \eta)  - \sigma(\sqrt{q'} \eta) } & \le& \nor{\sigma}_{\Lip} \abs{\sqrt{q} \eta - \sqrt{q'} \eta }  \le \nor{\sigma}_{\Lip} \nor{\sqrt{q} - \sqrt{q'} }_{\op}  \abs{\eta} \\
& \le & \nor{\sigma}_{\Lip} \sqrt{ \nor{{q} - {q'} }_{\op} } \abs{\eta}. \end{eqnarray*}
Therefore, 
$$\nor{  \sigma( \sqrt{q} \eta )^{\otimes 2 } } = \abs{ \sigma( \sqrt{q} \eta ) }^2 \le  \nor{\sigma}_{\Lip}^2 \bra{1+ \sqrt{\nor{ q  }_{\op} }}^2  \bra{1+\abs{\eta}}^2,$$
and
\begin{equation}\label{eq:continuity-sigma-eta}\begin{split}
 & \nor{  \sigma( \sqrt{q} \eta )^{\otimes 2 } - \sigma(\sqrt{q'} \eta)^{\otimes 2}}    \\
& \le \nor{  \sigma( \sqrt{q} \eta )^{\otimes 2}  - \sigma(\sqrt{q'} \eta)\otimes \sigma( \sqrt{q} \eta ) } + \nor{  \sigma( \sqrt{q'} \eta )^{\otimes 2}  - \sigma(\sqrt{q'} \eta)\otimes \sigma( \sqrt{q} \eta ) } \\
& = \bra{ \abs{  \sigma( \sqrt{q} \eta )} +\abs{  \sigma( \sqrt{q'} \eta )} }  \abs{  \sigma( \sqrt{q} \eta )-\sigma(\sqrt{q'} \eta)}\\
& \le  \nor{\sigma}_{\Lip}^2 \bra{2+ \sqrt{\nor{ q  }_{\op} } + \sqrt{\nor{ q'  }_{\op} }} \sqrt{ \nor{ q' -q}_{\op} } \bra{1+\abs{\eta}} \abs{\eta}.
\end{split}\end{equation}

\textbf{Step \thestep: lower semicontinuity of $J^\sigma$.} \stepcounter{step}  Choosing $\eta = N$ a standard Gaussian random variable in $\R^k$, we find that 
\begin{eqnarray*}\Lambda^{\sigma}(\lambda|q) & = &\log\EE\sqa{ \exp\bra{ \tr( \lambda \sigma(\sqrt{q}N)^{\otimes 2 } } } \\
& \le& \log \EE\sqa{ \exp \bra{ \nor{\lambda} \nor{ \sigma(\sqrt{q} N )^{\otimes 2} } } } \\
& \le &\log \EE\sqa{\exp \bra{ \nor{\lambda} \nor{\sigma}_{\Lip}^2 \bra{1+ \sqrt{\nor{ q  }_{\op} }}^2 \bra{1+\abs{N}}^2 } } \le c_0
\end{eqnarray*}
for some $c_0 < \infty$, for every $(\lambda, q) \in \R^{k \times k}_{\sym} \times \R^{k \times k}_{\succeq 0}$ in a suitable neighbourhood of $(0,0)$, e.g., provided that
$$ \nor{\lambda} \nor{\sigma}_{\Lip}^2 \bra{1+ \sqrt{\nor{ q  }_{\op} }}^2 \le 1/4. $$
This proves that for any given $q \in \R^{k \times k}_{\succeq 0}$, the function $\Lambda^{\sigma}(\cdot|q)$ is proper. Since lower semicontinuity and convexity are standard properties of $\log$-MGFs, we deduce that also its Legendre-Fenchel transform $J^{\sigma}(\cdot|q)$ is proper, lower semicontinuous and convex. Establishing joint lower semicontinuity of $J^\sigma$ requires some extra care: given any sequence  $((v_n, q_n))_{n=1}^{\infty} \subseteq \R^{k\times k}_{\succeq 0} \times \R^{k\times k}_{\succeq 0}$ converging to $(v,q)$, we claim that the sequence of proper functions $(\Lambda^{\sigma}(\cdot| q_n) )_{n=1}^{\infty}$ epi-converges to $\Lambda^{\sigma}(\cdot| q)$, so that by Theorem \ref{thm:wijsman} the Legendre-Fenchel transform $(J^{\sigma}(\cdot|q_n))_{n=1}^{\infty}$ epi-converges to $J^{\sigma}(\cdot|q)$, and in particular, by  \eqref{eq:liminf},
$$ J^{\sigma}(v|q) \le \liminf_{n \to \infty} J^{\sigma}(v_n|q_n),$$
establishing  lower semicontinuity. To prove the epi-convergence claim, consider any sequence $(\lambda_n)_{n=1}^{\infty} \subseteq \R^{k \times k}_{\sym}$ converging to $\lambda$. By continuity of the matrix square root and of $\sigma$, we have that 
$$ \exp\bra{ \tr( \lambda_n \sigma(\sqrt{q_n} N)^{\otimes 2} )} \to \exp\bra{ \tr( \lambda \sigma(\sqrt{q} N)^{\otimes 2}  )} \quad \text{a.s.\ as $n \to \infty$.}$$
Fatou's lemma entails
$$ \EE\sqa{  \exp\bra{ \tr( \lambda \sigma(\sqrt{q} N)^{\otimes 2}  )} } \le \liminf_{n \to \infty} \EE\sqa{  \exp\bra{ \tr( \lambda_n \sigma(\sqrt{q_n} N)^{\otimes 2}  )} },$$
that, up to taking $\log$ on both sides, establishes \eqref{eq:liminf} for $f_n = \Lambda^{\sigma}(\cdot | q_n)$. To prove \eqref{eq:limsup}, we collect first a general inequality: for any $\lambda \in \R^{k\times k}_{\sym}$, any $q, q' \in \R^{k \times k}_{\succeq 0}$ and $p\in (1, \infty)$,
\begin{equation} \small\label{eq:limsup-key-bound}\begin{split} & \EE\sqa{ \exp\bra{ \frac 1 p \tr\bra{ \lambda \sigma(\sqrt{q'} N)^{\otimes 2}  } } } \\
& = \EE\sqa{ \exp\bra{ \frac 1 p \tr\bra{  \lambda \sigma(\sqrt{q} N)^{\otimes 2}  } + { \frac 1 p \tr\bra{  \lambda \sqa{ \sigma(\sqrt{q'} N)^{\otimes 2}  -  \sigma(\sqrt{q} N)^{\otimes 2} } } } } }  \\
& \le  \EE\sqa{  \exp\bra{ \tr\bra{  \lambda \sigma(\sqrt{q} N)^{\otimes 2}  }}}^{\frac 1p} 
\EE\sqa{ \exp\bra{  (p-1)^{-1} \tr\bra{ \lambda \sqa{  \sigma(\sqrt{q'} N)^{\otimes 2} -  \sigma(\sqrt{q} N)^{\otimes 2}  }  }  } }^{1-
\frac 1 p},
\end{split}\end{equation}
by H\"older's inequality. Using \eqref{eq:continuity-sigma-eta}, we bound from above 
\begin{equation*}\begin{split}
& \EE\sqa{ \exp\bra{  (p-1)^{-1} \tr\bra{ \lambda \sqa{  \sigma(\sqrt{q'} N)^{\otimes 2} -  \sigma(\sqrt{q} N)^{\otimes 2}  }  }  } } \\
& \quad \le  \EE\sqa{ \exp\bra{ (p-1)^{-1} \nor{\lambda} \nor{  \sigma(\sqrt{q'} N)^{\otimes 2} -  \sigma(\sqrt{q} N)^{\otimes 2}  } }} \\
& \quad \le \EE\sqa{ \exp\bra{ (p-1)^{-1} \nor{\lambda} \nor{\sigma}_{\Lip}^2 \bra{2+ \sqrt{\nor{ q  }_{\op} } + \sqrt{\nor{ q'  }_{\op} }}  \sqrt{ \nor{ q' -q}_{\op} } \bra{1+\abs{N}} \abs{N} } }.
\end{split}
\end{equation*}
We now specialize to the case $q':= q_n$ with $q_n \to q$ as $n \to \infty$, so that $\sup_{n} \nor{ q_n}_{\op} < \infty$ and $\nor{q_n -q}_{\op} \to 0$. We can find therefore exponents $(p_n)_{n=1}^{\infty}\subset (1, \infty)$ with  $p_n \downarrow 1$ and such that
\begin{equation}\label{eq:condition-p_n}  (p_n-1)^{-1} \nor{\lambda} \nor{\sigma}_{\Lip}^2 \bra{2+ \sqrt{\nor{ q_n  }_{\op} } + \sqrt{\nor{ q  }_{\op} }}  \sqrt{ \nor{ q_n -q}_{\op} } \to 0. \end{equation}
Therefore
$$  \EE\sqa{ \exp\bra{  \frac{\nor{\lambda} \nor{\sigma}_{\Lip}^2}{(p_n-1)} \bra{2+ \sqrt{\nor{ q  }_{\op} } + \sqrt{\nor{ q_n  }_{\op} }}  \sqrt{ \nor{ q_n -q}_{\op} } \bra{1+\abs{N}} \abs{N} } } \to 1,$$
by dominated convergence, because $\exp\bra{\frac 1 4 \bra{1+\abs{N}}\abs{N} }$ is integrable and domination holds if  $n$ is so large that the left hand side in \eqref{eq:condition-p_n} is smaller than $1/4$. Taking the $\log$ in \eqref{eq:limsup-key-bound}, with $q':= q_n$ and $p:=p_n$, we eventually find 
$$ \limsup_{n \to \infty} \Lambda^{\sigma}\bra{  \lambda/{p_n} | q_n} \le   \Lambda^{\sigma}\bra{  \lambda| q},$$
that is \eqref{eq:limsup} with the sequence $\lambda_n := \lambda/p_n$. The proof of the lower semi-continuity of $J^{\sigma}$ is completed.

\textbf{Step \thestep: compactness of the union of sub-level sets of $J^\sigma$.} \stepcounter{step} 
Given any compact set $\mathsf K \subseteq \R^{k\times k}_{\succeq 0}$, by choosing $\lambda = \eps_0 \Id$ with $\eps_0>0$ such that
$$ \eps_0 \nor{\sigma}_{\Lip}^2 \bra{1+ \sqrt{\sup_{q \in \mathsf K} \nor{ q  }_{\op} }}^2 \le 1/4$$
we deduce that, for every $v \in \R^{k \times k}_{\succeq 0}$ and $q \in \mathsf K$,
$$ J^{\sigma}(v|q) \ge \eps_0 \tr( v ) - c_0, $$
and therefore, for every $L< \infty$,
$$\bigcup_{q \in \mathsf K} \cur{ v \in \R^{k\times k}_{\succeq 0}\, :\,  J^\sigma(v|q) \le L} \subseteq \cur{ v\in \R^{k\times k}_{\succeq 0}\, : \, \eps_0 \tr( v ) \le  c_0 + L }.$$
So the union of sub-level sets is a compact set since it is closed (see Remark \ref{re2:22012026}) and contained in a compact. 

\textbf{Step \thestep: proof of the LDP. 
} \stepcounter{step} First, we notice that if $q_n =q$ identically, the thesis is a direct application of Cramér's theorem in $\R^{k(k-1)/2}$, since the variables 
$$V_n' := \frac 1 n \sum_{i=1}^n \sigma(\sqrt{q} N_i)^{\otimes 2}$$
are the empirical mean of an i.i.d.\ family with log-MGF $\Lambda^{\sigma}(\cdot|q)$, that is finite in a neighborhood of the origin.  Then, we argue that $(V_n)_{n=1}^{\infty}$ and $(V'_n)_{n=1}^{\infty}$ are exponentially equivalent with speed $a_n =n$. Precisely, we establish \eqref{eq:condition-exp-equiv-moment}, that in this setting reads
$$ \lim_{n\to \infty} \frac 1 n \log \EE\sqa{ \exp\bra{ n \alpha \nor{V_n - \bar{V}_n}} } = 0, \quad \text{for every $\alpha>0$.}$$
Using \eqref{eq:continuity-sigma-eta} we have
\begin{equation*}\begin{split}  \nor{ V_n - V'_n }  & \le \frac 1 n \sum_{i=1}^n \nor{ \sigma(\sqrt{q_n} N_i)^{\otimes 2} - \sigma(\sqrt{q} N_i)^{\otimes 2} }\\
& \le  \nor{\sigma}_{\Lip}^2 \bra{ 2+ \sqrt{\nor{q}_{\op}}+ \sqrt{ \nor{q_n}_{\op} } } \sqrt{ \nor{ q_n - q }_{\op} } \cdot \frac 1 n \sum_{i=1}^n \abs{N_i}(1+\abs{N_i}).
\end{split}\end{equation*}
Since the variables $(N_i)_{i=1}^n$ are independent, for every $\alpha>0$,  we have
\begin{equation*}\begin{split}  & \frac 1 n \log \EE\sqa{ \exp\bra{ n \alpha \nor{V_n - V'_n}} } \\
& \quad \le \log\EE\sqa{\exp\bra{  \alpha \nor{\sigma}_{\Lip}^2 \sqrt{ \nor{ q_n - q }_{\op} }\bra{ 2+ \sqrt{\nor{q}_{\op}}+ \sqrt{ \nor{q_n}_{\op} } }  \abs{N_1}(1+\abs{N_1}) }}.
\end{split}
\end{equation*}
As $n \to \infty$, we have that $\nor{q_n - q}_{\op} \to 0$ and therefore, if $n$ is sufficiently large, the right hand side is finite and, by dominated convergence, infinitesimal.
\end{proof}

When $\sigma = \sigma^{(\ell)}$ is the activation function at layer $\ell$, we simply write 
$$ \Lambda^{(\ell)}(\lambda|q) := \Lambda^{\sigma^{(\ell)}}(\lambda | q), \quad \text{and} \quad  J^{(\ell)}(v|q) := J^{\sigma^{(\ell)}}(v |q).$$

With this notation, we establish the following LDP for finitely many inputs. We emphasize that, although its proof has the same structure of \cite[Theorem 1]{macci2024large}, it exploits tools that on one hand allow to extend the result to the general class of Lipschitz activation functions (recall that \cite[Theorem 1]{macci2024large} is limited to bounded activations) and on the other hand  pave the way to the functional generalization. 

\begin{proposition}[LDP for kernels and output] \label{prop:ldp-kernel}
Fix a depth $L \ge 1$, input $d_{\din}$ and output  $d_{\out}\ge 1$ dimensions, Lipschitz activation functions $(\sigma^{(\ell)})_{\ell=1,\ldots, L}$, and consider a sequence of MLPs outputs $f^{(L)}_{(n)}: \R^{d_{\din}} \to \R^{d_{\out}}$ with hidden layers widths $n_\ell = n_\ell(n)$ such that 
$$\lim_{n \to \infty} \frac{n_\ell}{n} := \beta^{(\ell)} \in (0, \infty] \quad \text{for every $\ell = 1, \ldots, L-1$.}$$
Let $\cX := \cur{x_i}_{i=1}^k  \subseteq \R^{d_{\din}}$ be fixed input points and consider the random (post-activation) kernels $v^{(\ell)}(\cX)_{(n)}$ defined as in \eqref{eq:v-def}. Then, 
\begin{enumerate}
\item [(i)] for every $\ell=0, 1, \ldots, L-1$, the LDP holds for $(v^{(\ell)}_{(n)}(\cX))_{n=1}^\infty$ with speed $n$ and rate function recursively defined as $I^{(0)}(v) := \infty 1_{\cur{v\neq v^{(0)}(\cX) }}$ and
$$ I^{(\ell)}(v) := \inf_{q \in \R^{k \times k}_{\succeq 0}}\Big\{ \beta^{(\ell)} J^{(\ell)}(v | q) + I^{(\ell-1)}(q)\Big\}, \quad \text{for $\ell = 1, \ldots, L-1$};$$
\item [(ii)] the LDP holds for $\bra{ \frac 1 {\sqrt{n}} f^{(L)}_{(n)}(\cX)}_{n=1}^\infty$ with speed $n$ and rate function
\begin{equation}\label{eq:rate-out-finite} I^{(L)}_{\cX, \out}( (z_i)_{i=1}^k ) := \inf_{q \in \R^{k \times k}_{\succeq 0}} \Big\{\sum_{i=1}^k \frac 1 2  \nor{z_i}_q^2  + I^{(L-1)}(q)\Big\}.\end{equation}
\end{enumerate}
\end{proposition}

Let us notice that according to \eqref{eq:v-def}, the kernel $v^{(0)}(\cX)$ is deterministic and does not depend on $n$, hence $I^{(0)}$ is well-defined.

\begin{proof}
To prove $(i)$, i.e.,  the LDP for kernels, we argue by induction over $\ell$, the case $\ell=0$ being trivial, because the kernel is deterministic. Assuming that the thesis holds for $\ell-1$, to prove it for the case $\ell$ we start from the observation in Remark \ref{rem:conditional-gaussian}, so that conditionally upon $f^{(\ell-1)}_{(n)}(\cX)$, we have
$$ v^{(\ell)}_{(n)}(\cX) \stackrel{law}{=} \frac 1 {n_\ell} \sum_{i=1}^{n_\ell} \sigma\bra{ \sqrt{ v^{(\ell-1)}_{(n)}(\cX) } N_i }^{\otimes 2 }.$$
By inductive assumption, we have that $v^{(\ell-1)}_{(n)}(\cX)$ satisfies the LDP with speed $n$ and rate function
$I^{(\ell-1)}(v).$
Moreover, by Lemma~\ref{lem:cramer-best} and Remark \ref{re:22012026}, the LDP continuity condition holds for the sequence $(P(v^{(\ell)}_{(n)}(\cX)=\cdot|v^{(\ell-1)}_{(n)}(\cX)=\cdot))_{n=1}^{\infty}$, with speed $n$ and function $\beta^{(\ell)} J^{(\ell)}$. Therefore, by Chaganty's lemma we deduce that the LDP holds for the sequence $((v^{(\ell)}_{(n)}(\cX),v^{(\ell-1)}_{(n)}(\cX)))_n$ with speed $n$ and rate function
\[
(v,q)\mapsto \beta^{(\ell)} J^{(\ell)}(v | q) + I^{(\ell-1)}(q).
\]
We conclude with a simple application of
the contraction principle. 

For $(ii)$, conditionally on $v^{(L-1)}_{(n)}(\cX)=q$, the output $f^{(L)}_{(n)} (\cX)$ consists of $d_{\out}$ i.i.d.\ $k$-dimensional Gaussian random variables, each with zero mean and covariance matrix $q$, i.e., 
\[
\frac1{\sqrt n}f^{(L)}(\cX) \stackrel{law}{=}   \bra{ \frac{  \sqrt{q}N_i }{\sqrt{n} }}_{i=1, \ldots, d_{\out} }.
\]
Recalling Example~\ref{ex:gaussian-chaganty} (slightly extended to cover the case $k \ge 1$) we obtain \eqref{eq:rate-out-finite} concluding the proof. 
\end{proof}

We now explain how the finite-input large deviation principle established in the previous proposition can be lifted to a functional large deviation principle for the entire process, thus completing the proof of Theorem~\ref{thm:main}. The argument follows the classical projective limit approach based on the Dawson--G\"artner theorem combined with exponential tightness.

Let $\cX \subseteq \R^{d_{\din}}$ denote the compact input set, and, for each finite subset $\cX_0 \subseteq \cX$, let $\pi_{\cX_0}$ be the canonical projection mapping the output process to its restriction on $\cX_0$. We have shown that, for every such finite set $\cX_0$, the projected process $\pi_{\cX_0}(Z_n)$,  where $Z_n=\frac{1}{\sqrt{n}} f^{(L)}_{(n)}$, $n\geq 1$, satisfies a large deviation principle with a rate function $I^{(L)}_{\cX_0, \out}$ defined in \eqref{eq:rate-out-finite}. By contraction principle, this entails automatically that these rate functions are consistent under further projections. In addition, exponential tightness of 
the family $(Z_n)_{n=1}^{\infty}$ in the function space $C(\cX; \R^k)$ has been established in Section~\ref{sec:exp-tight}, so that Proposition \ref{prop:dawson} applies and yields a large deviation principle for the sequence $(Z_n)_n$ with speed $n$ and rate function given by
\begin{equation}\label{eq:rate-functional}
I_{\cX, \out}^{(L)} (z) = \sup_{\cX_0 \subset\cX \text{finite}} I_{\cX_0, \out}^{(L)}\bra{\pi_{\cX_0}(z)}.
\end{equation}

\section{Acknowledgements}

C.M.\ and B.P.\ acknowledge the partial support of MUR Excellence Department Project awarded to the Department of Mathematics, University of Rome Tor Vergata (CUP E83C23000330006), and of University of Rome Tor Vergata (CUP E83C25000630005) Research Project METRO. 
D.T.\ acknowledges the MUR Excellence Department Project awarded to the Department of Mathematics, University of Pisa, CUP I57G22000700001, the HPC Italian National Centre for HPC, Big Data and Quantum Computing - Proposal code CN1 CN000\-00\-013, CUP I53C22000690001, the PRIN 2022 Italian grant 2022{-}WHZ5XH - ``understanding the LEarning process of QUantum Neural networks (LeQun)'', CUP J53D23003890006, the project  G24-202 ``Variational methods for geometric and optimal matching problems'' funded by Università Italo Francese.  D.T.\ and K.P.\ acknowledge the partial support of the project PNRR - M4C2 - Investimento 1.3, Partenariato Esteso PE00000013 - ``FAIR - Future Artificial Intelligence Research'' - Spoke 1 ``Human-centered AI'', funded by the European Commission under the NextGeneration EU programme. 
G.L.T. \ acknowledges the financial support of ERA4TB project,
grant agreement ID: 853989, and Praesiidium project, HORIZON-HLTH-2022-STAYHLTH-02.

This research benefitted from the support of INdAM-GNAMPA and the FMJH Program Gaspard Monge for optimization and operations research and their interactions with data science. 

\appendix 

\section{Proof of Theorem~\ref{thm:schied}}\label{app:schied}

First, we recall the following quantitative version of Kolmogorov's continuity theorem, whose proof is obtained by constant chasing in the standard chaining-based argument, see e.g.\ \cite[Chapter I, Theorem 2.1]{revuz2013continuous}, or \cite[Theorem 14.1.1]{marcus2006markov}.

\begin{theorem}
\label{thm:kolmogorov}
Let $(Z(x))_{x\in[0,1]^d}$ be an $\R^k$-valued continuous random field and suppose that there exist $\bar n\geq 1$, $\kappa,\alpha>0$ such that
\[
\|Z(x)-Z(y)\|_{L^n}\leq \kappa |x-y|^{\alpha},\quad\text{$\forall$ $n\geq \bar n$,  $x,y\in [0,1]^d$.}
\]
Then,  for each $\alpha'\in (0,\alpha)$,  there exist $n'=n'(\alpha,\alpha',d)$ and
a constant $\cost=\cost (\alpha,\alpha',\kappa,d)\in (0,\infty)$ such that
\[
\|M_{\alpha'}\|_{L^n}\leq\cost,\quad\text{$\forall$ $n\geq\max\{\bar n,n'\}$}
\]
where
\[
M_{\alpha'}:=\sup_{x,y\in [0,1]^d,x\neq y}\frac{|Z(x)-Z(y)|}{|x-y|^{\alpha'}}.
\]
A close inspection of the proof yields
\[
n':=\lceil d/(\alpha-\alpha')\rceil\quad\text{and}\quad\cost:=\frac{\kappa 2^{\alpha'+1+d}}{1-2^{-(\alpha-\alpha'-\frac{d}{n'})}},
\]
where $\lceil\cdot\rceil$ denotes the ceiling function.
\end{theorem}

\begin{proof} 
Since all the norms on $[0,1]^d$ are equivalent, without loss of generality we prove the result considering on $[0,1]^d$ the norm
$|t|=\max_{1\leq i\leq d}|t_i|$. 
 
For $m\in\mathbb N$,  let $D_m$ be the set of $d$-uples $2^{-m}(i_1,\ldots,i_d)$ with 
$i_k\in [0,2^{m})$ and put $D:=\bigcup_m D_m$.  Note that $D$ is dense in $[0,1]^d$ and so by the continuity of the random field we have
\[
\sup_{x,y\in [0,1]^d,x\neq y}\frac{|Z(x)-Z(y)|}{|x-y|^{\alpha'}}=
\sup_{x,y\in D,x\neq y}\frac{|Z(x)-Z(y)|}{|x-y|^{\alpha'}}.
\]
Let $\Delta_m:=\{(x,y)\in D_m^2:\,\,|x-y|=2^{-m}\}$.  It turns out that $|\Delta_m|\leq 2^{(m+1)d}$,  being $|\Delta_m|$ the cardinality of $\Delta_m$.  Setting
\[
K_i:=\sup_{(x,y)\in\Delta_i}|Z(x)-Z(y)|,
\]
we have
\[
K_i^n=\sup_{(x,y)\in\Delta_i}|Z(x)-Z(y)|^n\leq\sum_{(x,y)\in\Delta_i}|Z(x)-Z(y)|^n
\]
and so
\begin{align}
\|K_i\|_{L^n}^n&\leq\sum_{(x,y)\in\Delta_i}\|Z(x)-Z(y)\|_{L^n}^n\nonumber\\
&\leq \kappa^n\sum_{(x,y)\in\Delta_i}|x-y|^{n\alpha}\leq \kappa^n2^{-in\alpha}2^{(i+1)d}=\kappa^n 2^d 2^{-(n\alpha-d)i},\quad\text{$\forall$ $n\geq \bar n$.}\nonumber
\end{align}
Thus
\begin{align}
\|K_i\|_{L^n}&\leq \kappa 2^{d} 2^{-(\alpha-\frac{d}{n})i},\quad\text{$\forall$ $n\geq \bar n$.}\label{eq:2}
\end{align}

Let $x,y\in D$ be such that $|x-y|\leq 2^{-m}$; arguing as in [RY],  Theorem 2.1,  Chapter 1, we have
\begin{equation}\label{eq:1}
|Z(x)-Z(y)|\leq 2\sum_{i=m}^{\infty}K_i.
\end{equation}
Note that for $x,y\in D$ it holds $|x-y|=\max_{1\leq i\leq d}|x_i-y_i|\leq 1$,  and so 
\[
D^2\subseteq\bigcup_{m\geq 0}\widetilde{D}_m
\]
where $\widetilde D_m:=\{(x,y)\in D^2:\,\,2^{-(m+1)}\leq|x-y|\leq 2^{-m}\}$.  Therefore,  using \eqref{eq:1} one has
\begin{align}
M_{\alpha'}&\leq\sup_{m\in\mathbb N}\sup_{x,y\in\widetilde{D}_m:\,\,x\neq y}\frac{|Z(x)-Z(y)|}{|x-y|^{\alpha'}}\nonumber\\
&\leq\sup_{m\in\mathbb N}2^{(m+1)\alpha'}\sup_{|x-y|\leq 2^{-m},\,\,x,y\in D}|Z(x)-Z(y)|\nonumber\\
&\leq\sup_{m\in\mathbb N}2^{(m+1)\alpha'+1}
\sum_{i=m}^{\infty}K_i\nonumber\\
&=\sup_{m\in\mathbb N}2^{\alpha'+1}
\sum_{i=m}^{\infty}2^{m\alpha'}K_i\nonumber\\
&\leq 2^{\alpha'+1}\sum_{i=0}^{\infty}2^{\alpha' i}K_i.\nonumber
\end{align}
Combining this inequality with \eqref{eq:2},  we have
\begin{align}
\|M_{\alpha'}\|_{L^n}\leq2^{\alpha'+1}\sum_{i=0}^{\infty}2^{\alpha' i}\|K_i\|_{L^n}
\leq \kappa 2^{\alpha'+1+d}\sum_{i=0}^{\infty}2^{-(\alpha-\alpha'-\frac{d}{n})i}
\quad\text{$\forall$ $n\geq \bar n$.}
\nonumber
\end{align}
Setting $n':=\lceil d/(\alpha-\alpha')\rceil$,  we finally have
\begin{equation*}
\|M_{\alpha'}\|_{L^n}\leq \kappa 2^{\alpha'+1+d}\sum_{i=0}^{\infty}2^{-(\alpha-\alpha'-\frac{d}{n_1})i}:=\cost,
\quad\text{$\forall$ $n\geq\max\{\bar n,n'\}$.}
\nonumber \qedhere
\end{equation*}
\end{proof}

We use Theorem~\ref{thm:kolmogorov} to prove Theorem~\ref{thm:schied}. 
We recall first that by inequality \eqref{eq:schied-increments}, for every $n \ge \bar n$, we have
$$ \EE\sqa{ \abs{Z_n(x) - Z_n(y)}^{n} } \le \cost^n |x-y|^{n \alpha} \quad \text{for every $x, y \in [0,1]^d$.}$$
Then, for $\alpha'\in (0,\alpha)$, by Theorem~\ref{thm:kolmogorov} there exist $n'=n'(\alpha,\alpha',d)$ and $\cost'=\cost'(\alpha,\alpha',\cost,d)$ such that
\[
\EE\sqa{ \bra{  \sup_{x\neq y \in [0,1]^d } \frac{\abs{Z_n(x)- Z_n(y)}}{|x-y|^{\alpha'} } }^{n} } \le (\cost')^{n},\quad\text{$\forall$ $n\geq \max\{\bar n,n'\}$}.
\]
Moreover, by \eqref{eq:schied-zero} we easily have that  for $n \ge\max\{\bar n, n'\}$, 
$$  \EE\sqa{ \bra{1 + |Z_n(0)|}^{n} } \le \cost^{n}.$$
We put together these bounds with the inequality $1+a+b \le (1+a)(1+b)$, $a,b\ge 0$, to find
\begin{equation*}\begin{split}
& \EE\sqa{ \bra{ 1 +  \abs{Z_n(0)} + \sup_{x\neq y \in [0,1]^d } \frac{\abs{Z_n(x)- Z_n(y)}}{|x-y|^{\alpha'} } }^{n/2} } \\
& \le \EE\sqa{ \bra{ 1 +  \abs{Z_n(0)}}^{n/2} \bra{1+\sup_{x\neq y \in [0,1]^d } \frac{\abs{Z_n(x), Z_n(y)}}{|x-y|^{\alpha'} } }^{n/2} }\\
& \le  \EE\sqa{ \bra{ 1 +  \abs{Z_n(0)}}^{n}}^{1/2} \EE\sqa{\bra{1+\sup_{x\neq y \in [0,1]^d } \frac{\abs{Z_n(x)- Z_n(y)}}{|x-y|^{\alpha'} } }^{n}}^{1/2}\\
& \le[\cost(1+\cost')]^{n/2},\quad\quad\quad\text{$\forall$ $n\ge\max\{\bar n, n'\}$}
\end{split}\end{equation*}
which leads to exponential tightness in the form
$$ \sup_{ n \ge \max\{\bar n,\tilde n\} } \frac 1 {n} \log\EE\sqa{ \psi( Z_n )^{n} } < \infty,$$
where $\psi: C([0,1]^d; \R^k) \to [0, \infty]$ denotes the function
$$ \psi( z ) = \bra{ 1+ \abs{z(0)} +\sup_{x\neq y \in [0,1]^d } \frac{\abs{z(x)- z(y)}}{|x-y|^{\alpha'}}}^{1/2} $$ 
which has compact sub-level sets (with respect to the uniform topology) as a consequence of Arzel\`a-Ascoli theorem (indeed $\psi$ is lower semicontinuous with respect to the sup-norm).

\printbibliography

\end{document}